\documentclass[12pt]{amsart}
\usepackage{times}
\usepackage{amsfonts}
\usepackage{amssymb}
\usepackage{bbm}
\usepackage{times}
\usepackage{amssymb}
\usepackage{amscd}
\usepackage{graphicx}

\usepackage{amsmath}
\usepackage{amssymb}

\usepackage{amsmath}
\usepackage{amsfonts}
\usepackage{amscd}
\usepackage{latexsym}
\usepackage{amsthm}
\usepackage{amssymb}
\usepackage{amsmath}
\usepackage{setspace}
\theoremstyle{plain}
\newtheorem{theorem}{Theorem}[section]
\newtheorem{corollary}[theorem]{Corollary}

\newtheorem{definition}[theorem]{Definition}

\def\ep{\varepsilon}

\begin{document}

\vskip 0.5cm

\title[generalized tracially approximated ${\rm C^*}$-algebras] {generalized tracially approximated ${\rm C^*}$-algebras}
\author{George A. Elliott, Qingzhai Fan, and Xiaochun Fang}

\address{George A. Elliott\\ Department of Mathematics\\ Univesity of Toronto\\ Toronto\\ Ontario \\ Canada  \hspace{0.1cm} M5S~ 2E4}
\email{elliott@math.toronto.edu}

\address{Qingzhai Fan\\ Department of Mathematics\\  Shanghai Maritime University\\
Shanghai\\China
\\  201306 }
\email{qzfan@shmtu.edu.cn}

\address{Xiaochun Fang\\ Department of Mathematics\\ Tongji
University\\
Shanghai\\China
\\200092}
\email{xfang@tongji.edu.cn}

\thanks{2000 \emph{Mathematics Subject Classification\rm{:}} 46L35, 46L05, 46L80}

\begin{abstract} In this paper, we introduce  some classes of  generalized  tracial approximation  ${\rm C^*}$-algebras.  Consider the class of  unital ${\rm C^*}$-algebras which   are tracially $\mathcal{Z}$-absorbing (or have tracial nuclear dimension at most $n$, or have the property $\rm SP$, or are  $m$-almost divisible). Then $A$  is  tracially $\mathcal{Z}$-absorbing  (respectively, has tracial nuclear dimension at most $n$, has the property $\rm SP$, is weakly ($n, m$)-almost divisible) for  any  simple  unital ${\rm C^*}$-algebra $A$ in the corresponding  class of generalized tracial approximation ${\rm C^*}$-algebras.
As an application, let $A$ be an infinite-dimensional unital simple ${\rm C^*}$-algebra, and let $B$ be a centrally large subalgebra of $A$. If $B$ is  tracially $\mathcal{Z}$-absorbing, then $A$ is  tracially $\mathcal{Z}$-absorbing. This result  was obtained by  Archey, Buck, and  Phillips in \cite{AJN}.
\end{abstract}

\maketitle

\section{Introduction}
The Elliott program for the classification of amenable
 ${\rm C^*}$-algebras might be said to have begun with the ${\rm K}$-theoretical
 classification of AF algebras in \cite{E1}. A major next step was the classification of simple
  AH algebras without dimension growth (in the real rank zero case see \cite{E6}, and in the general case
 see \cite{GL}).
  This led eventually to the classification of  simple separable amenable ${\rm C^*}$-algebras with finite nuclear dimension in the UCT class (see \cite{KP}, \cite{PP2}, \cite{EZ5}, \cite{GLN1}, \cite{GLN2}, \cite{TWW1}, \cite{EGLN1},  \cite {GL2}, and \cite{GL3}).

A crucial intermediate step was Lin's axiomatization of  Elliott-Gong's decomposition theorem for  simple AH algebras of real rank zero (classified by Elliott-Gong in \cite{E6}) and Gong's decomposition theorem (\cite{G1}) for simple AH algebras (classified by Elliott-Gong-Li in \cite{GL}). For this purpose, Lin introduced the concepts of  TAF and TAI (\cite{L0} and \cite{L1}). (A weaker version of the property TAF had been introduced by Popa in \cite{PO}.) Instead of assuming inductive limit structure, Lin started with a certain abstract (tracial) approximation property. Elliott and  Niu in \cite{EZ} considered  this  notion of
 tracial approximation by  other classes of unital ${\rm C^*}$-algebras than the finite-dimensional ones  for TAF and the interval algebras for TAI.
  In \cite{EZ}, Elliott and Niu, and in \cite{EFF}, Elliott, Fan, and Fang  showed that certain properties of  ${\rm C^*}$-algebras in a given class $\Omega$ are inherited by a simple unital ${\rm C^*}$-algebra
in the class ${\rm TA}\Omega$.

 Large and centrally  large subalgebras were introduced in \cite{P3} and \cite{AN} by Phillips and  Archey   as  abstractions of Putnam's orbit breaking subalgebra of the crossed product algebra ${\rm C}^*(X,\mathbb{Z},\sigma)$ of the Cantor set by a minimal homeomorphism in \cite{P}.

In \cite{AN}, Archey and  Phillips showed that  if $B$ is centrally large in $A$ and $B$ has stable rank one, then so  also does $A$.
 In \cite{AJN},  Archey, Buck, and  Phillips  proved that if $A$ is a simple infinite-dimensional stably finite unital  ${\rm C^*}$-algebra and $B\subseteq A$ is a centrally  large subalgebra, then $A$ is tracially $\mathcal{Z}$-absorbing  in the sense of  \cite{HO} if, and only if, $B$ is tracially $\mathcal{Z}$-absorbing.

Inspired  by  centrally large subalgebras and tracial approximation ${\rm C^*}$-algebras, we introduce a class of  generalized   tracial approximation ${\rm C^*}$-algebras.  The notion  generalizes  both Archey and Phillips's centrally large subalgebras and Lin's notion of tracial approximation.

Let  $\Omega$ be a class of unital ${\rm C^*}$-algebras. We define as follows  the class
 of ${\rm C^*}$-algebras which can be weakly  tracially approximated by ${\rm C^*}$-algebras in $\Omega$, and denote this class by ${\rm WTA}\Omega$.

\begin{definition}\label{def:1.1}  A  simple unital ${\rm C^*}$-algebra $A$ will be  said to belong to the class ${\rm WTA}\Omega$  if, for any
 $\varepsilon>0$, any finite
subset $F\subseteq A$, and any  non-zero element $a\geq 0$, there
exist a  projection $p\in A$, an  element $g\in A$ with $0\leq g\leq 1$,
  and a unital ${\rm C^*}$-subalgebra $B$ of $A$ with
$g\in B, 1_B=p$, and $B\in \Omega$, such that

$(1)$  $(p-g)x\in_{\varepsilon} B, ~ x(p-g)\in_{\varepsilon} B$, for all $x\in  F$,

$(2)$ $\|(p-g)x-x(p-g)\|<\varepsilon$, for all $x\in F$,

$(3)$ $1-(p-g)\precsim a$ (see Section 2), and

$(4)$ $\|(p-g)a(p-g)\|\geq \|a\|-\varepsilon$.
\end{definition}

It  follows  from the definitions and by the proof of Theorem 4.1 of \cite{EZ} that if $A$ is a simple   unital  ${\rm C^*}$-algebra and  $A\in {\rm  TA}\Omega$ (Definition \ref{def:2.5}), then $A\in {\rm WTA}\Omega$. Furthermore, if $\Omega=\{B\}$, and $B\subseteq A$ is a centrally large subalgebra of $A$ (Definition \ref{def:2.6}), then $A\in {\rm WTA}\Omega$.

 In Theorem 3.9 of  \cite{N2}, Niu shows that if $(X,\sigma,\Gamma)$ is  a dynamical system  ($X$ compact metrizable and $\Gamma$ countable amenable) with the (URP), then  the crossed product ${\rm C^*}$-algebra can be weakly tracially approximated in a non-unital sense by (not necessarily unital) homogeneous ${\rm C^*}$-algebras with dimension ratio almost dominated by the mean dimension of $(X,\sigma,\Gamma)$. Let $\Omega= \{
C^{\dag} : C \cong \bigoplus_{s=1}^N \mathrm{M}_{K_s} (\rm{C}_0(Z_s)) ,
\, N \in \mathbb{Z}_{\geq 0},\}$,   where $\text{each } Z_s$   is a locally compact Hausdorff space,  and  $C^{\dag}$ is the unitization of $C$.
 By the proof of Theorem 3.9 of  \cite{N2},  one can   show that  Niu's crossed product ${\rm C^*}$-algebra  belongs to the class ${\rm WTA}\Omega$.

The Rokhlin property in ergodic theory was adapted to
the context
 of von Neumann algebras by  Connes in \cite{AC}. It was adapted by  Herman
 and  Ocneanu  for UHF-algebras in \cite{HHO1}. In  \cite{K3} Kishimoto, in \cite{M1} and \cite{M2} Izumi, and in \cite{R11} R{\o}rdam
 considered the Rokhlin property in  a much more general ${\rm C^*}$-algebra context. More recently,  Osaka and Phillips  studied
actions of a  finite group and of the  group $\mathbb{Z}$ of integers on  certain simple ${\rm C^*}$-algebras with a modified  Rokhlin property
  in \cite{OP} and \cite{P2}.

   In  \cite{OP}, Osaka and Phillips showed  the following  result (a part of Lemma 2.5 of \cite{OP}): Let $A$ be a stably finite simple unital $C^*$-algebra with real rank zero such that the order on projections in matrix algebras  over $A$ is determined by traces (strict comparison of projections).
 Let $\alpha\in {\rm Aut}(A)$ have   the tracial Rokhlin property. Then for every finite set $F\subseteq {\rm C}^*(\mathbb{Z}, A, \alpha)$, every $\varepsilon>0$, every $N\in \mathbb{N}$,
  and every non-zero positive element $z\in A$, there exist a projection $p\in A$, a unital subalgebra $D\subseteq p{\rm C}^*(\mathbb{Z}, A, \alpha)p$ with $1_D=p$,  and a projection $q\in D$ such that

$(1)$ $qx\in_\ep D, ~ xq\in_\ep D$, for
all $x\in  F$,

 $(2)$ $1-q\precsim z$, and

 $(3)$ $ N\langle 1-q\rangle\leq \langle q\rangle$ (see Section 2).

Inspired by   this result, in this paper, we introduce another class of  generalized tracial approximation ${\rm C^*}$-algebras. Let  $\Omega$ be a class of unital C$^*$-algebras. We define as follows  the class
 of C$^*$-algebras which can be   tracially almost
approximated by C$^*$-algebras in $\Omega$, and denote this class by TAA$\Omega$.
\begin{definition} \label{def:1.3} A  simple unital ${\rm C^*}$-algebra $A$ will be  said to belong to the class ${\rm TAA}\Omega$ if, for any $\varepsilon>0$, any finite
subset $F\subseteq A$,  any  non-zero element $a\geq 0$,  there
exist a  projection $p\in A$,   a ${\rm C^*}$-subalgebra $B$ of $A$ with
$1_B=p$ and $B\in \Omega$, and  a projection $q\in B$  such that

$(1)$  $qx\in_{\varepsilon} B,~ xq\in_{\varepsilon} B$, for all $x\in  F$,

$(2)$ $1-q\precsim a$ (see Section 2), and

$(3)$ $\|qaq\|\geq \|a\|-\varepsilon$.

\end{definition}
In this paper, we shall prove the following five  results:

 Let $\Omega$ be a class of  unital  ${\rm C^*}$-algebras with the property $\rm SP$.  Then  $A$  has the  property $\rm SP$  for  any  simple unital  ${\rm C^*}$-algebra $A\in{\rm  WTA}\Omega.$  (Theorem \ref{thm:3.1}.)

 Let $\Omega$ be a class of   unital
${\rm C^*}$-algebras which   are tracially $\mathcal{Z}$-absorbing
(Definition \ref{def:2.3}).   Then  $A$  is  tracially $\mathcal{Z}$-absorbing  for  any   simple unital  ${\rm C^*}$-algebra $A\in{\rm  WTA}\Omega.$ (Theorem \ref{thm:3.4}.)
\vskip 0.1cm
  Let $\Omega$ be a class of   unital
${\rm C^*}$-algebras with tracial nuclear dimension at most $n$
 (Definition \ref{def:2.4}).  Then  $A$  has tracial nuclear dimension at most $n$ for  any  simple  unital ${\rm C^*}$-algebra $A\in{\rm  WTA}\Omega.$ (Theorem \ref{thm:3.7}.)
\vskip 0.1cm
Let $\Omega$ be a class of   unital
${\rm C^*}$-algebras  which are  $m$-almost divisible (Definition \ref{def:2.7}).   Let  $A\in{\rm WTA}\Omega$ be a  simple unital   stably finite ${\rm C^*}$-algebra such that for any $n\in \mathbb{N}$ the
 ${\rm C^*}$-algebra  $\rm{M}$$_n(A)$  belongs to the class ${\rm WTA}\Omega$. Then $A$  is weakly $(2, m)$-almost divisible (Definition \ref{def:2.8}). (Theorem \ref{thm:3.10}.)
\vskip 0.1cm
  Let $\Omega$ be a class of   unital
${\rm C^*}$-algebras  which are  $m$-almost divisible.   Let  $A\in{\rm TAA}\Omega$ be a simple unital stably finite  ${\rm C^*}$-algebra such that for any $n\in \mathbb{N}$ and any  unital hereditary
 ${\rm C^*}$-subalgebra $D$ of  $\rm{M}$$_n(A)$, $D$  belongs to the class ${\rm TAA}\Omega$. Then
   $A$  is weakly $(2, m)$-almost divisible. (Theorem \ref{thm:3.11}.)
\vskip 0.1cm

As applications, the following known results follow from these results.

 Let $A$ be a simple  unital ${\rm C^*}$-algebra, and let $B$ be a centrally large
 subalgebra of $A$. If $B$ is tracially $\mathcal{Z}$-absorbing, then $A$ is tracially $\mathcal{Z}$-absorbing. This result was obtained by  Archey, Buck, and  Phillips in \cite{AJN}.

 Let $\Omega$ be a class of   unital
${\rm C^*}$-algebras which are tracially $\mathcal{Z}$-absorbing. Then $A$ is tracially $\mathcal{Z}$-absorbing  for  any  simple  unital ${\rm C^*}$-algebra $A\in {\rm TA}\Omega$. This result was obtained by Elliott, Fan, and Fang in \cite{EFF}.

 Let $A$ be a simple  unital ${\rm C^*}$-algebra, and let $B$ be a centrally large
 subalgebra of $A$. If $B$ has tracial nuclear dimension at most $n$, then $A$ has  tracial nuclear dimension at most $n$. This result was obtained by  Zhao, Fang,  and Fan in \cite{ZF}

Let $\Omega$ be a class of    unital
${\rm C^*}$-algebras which have tracial nuclear dimension at most $n$. Then $A$ has tracial nuclear dimension at most $n$  for  any  simple  unital ${\rm C^*}$-algebra $A\in {\rm TA}\Omega$. This result was obtained by Fan and Yang in \cite{Q9}.

\section{Preliminaries and definitions}

Recall that  a $\rm C^*$-algebra  $A$ has the  property $\rm SP$ if every non-zero
hereditary $\rm C^*$-subalgebra of $A$ contains a non-zero projection.

Let $A$ be a ${\rm C^*}$-algebra, and let ${\rm M}_n(A)$ denote the ${\rm C^*}$-algebra of  $n\times n$ matrices with entries
elements of $A$. Let ${\rm M}_{\infty}(A)$ denote the algebraic inductive  limit of the sequence  $({\rm M}_n(A),\phi_n),$
where $\phi_n:{\rm M}_n(A)\to {\rm M}_{n+1}(A)$ is the canonical embedding as the upper left-hand corner block.
 Let ${\rm M}_{\infty}(A)_+$ (respectively, ${\rm M}_{n}(A)_+$) denote
the positive elements of ${\rm M}_{\infty}(A)$ (respectively, ${\rm M}_{n}(A)$).  Given $a, b\in {\rm M}_{\infty}(A)_+,$
one says  that $a$ is Cuntz subequivalent to $b$ (written $a\precsim b$) if there is a sequence $(v_n)_{n=1}^\infty$
of elements of ${\rm M}_{\infty}(A)$ such that $$\lim_{n\to \infty}\|v_nbv_n^*-a\|=0.$$
One says that $a$ and $b$ are Cuntz equivalent (written $a\sim b$) if $a\precsim b$ and $b\precsim a$. We  shall write $\langle a\rangle$ for the Cuntz equivalence class of $a$.

The object ${\rm Cu}(A):={(A\otimes {{K}}})_+/\sim$
 will be called the Cuntz semigroup of $A$. (See \cite{CEI}.)  Observe that  any $a, b\in {\rm M}_{\infty}(A)_+$
are Cuntz equivalent  to orthogonal  elements $a', b'\in {\rm M}_{\infty}(A)_+$ (i.e., $a'b'=0$),   and so ${\rm Cu}(A)$ becomes  an ordered  semigroup   when equipped with the addition operation
$$\langle a\rangle+\langle b\rangle=\langle a+ b\rangle$$
 whenever $ab=0$, and the order relation
$$\langle a\rangle\leq \langle b\rangle\Leftrightarrow a\precsim b.$$

Given $a$ in ${\rm M}_{\infty}(A)_+$ and $\varepsilon>0,$ we denote by $(a-\varepsilon)_+$ the element of ${\rm C^*}(a)$ corresponding (via the functional calculus) to the function $f(t)={\max (0, t-\varepsilon)},~~ t\in \sigma(a)$. By the functional calculus, it follows in a straightforward manner that $((a-\varepsilon_1)_+-\varepsilon_2)_+=(a-(\varepsilon_1+\varepsilon_2))_+.$

Let $0<\varepsilon<1$ be two positive numbers. Define
$$
f_{\varepsilon}(t)=\left\{
  \begin{array}{ll}
   1 & if\ t\geq\varepsilon,\\
  {(2t-\varepsilon)}/{\varepsilon} & if\ \varepsilon/2< t \leq \varepsilon,\\
  0 & if\ 0\leq t\leq \varepsilon/2.
  \end{array} \right.
 $$
The following facts are  well known.
\begin{theorem}(\cite{PPT}, \cite{HO}, \cite{P3}, \cite{RW}.) \label{thm:2.1} Let $A$ be a ${\rm C^*}$-algebra.

 $(1)$ Let $a,~ b\in A_+$ and   $\varepsilon>0$  be such that
$\|a-b\|<\varepsilon$.  Then there is a contraction $d$ in $A$ with $(a-\varepsilon)_+=dbd^*$.

$(2)$ Let $a,~ p$ be positive elements in ${\rm M}_{\infty}(A)$ with $p$ a projection. If $p\precsim a,$ then there is $b$ in ${\rm M}_{\infty}(A)_+$ such that  $bp=0$ and $b+p\sim a$.

 $(3)$ Let $a$ be a  positive element  of $A$
not Cuntz equivalent to a projection. Let  $\delta>0$,  and let $f\in C_0(0,1]$ be a non-negative function with $f=0$ on $(\delta,1),$  $f>0$ on $(0,\delta)$,
and $\|f\|=1$.  Then $f(a)\neq 0$
and  $(a-\delta)_++f(a)\precsim a.$

$(4)$ Let $a, b\in A$ satisfy $0\leq a\leq b$. Let $\varepsilon\geq 0$.
Then $(a-\varepsilon)_+\precsim(b-\varepsilon)_+$ (Lemma 1.7 of \cite{P3}).
\end{theorem}

Winter and Zacharias  introduced the notion of  nuclear dimension for  ${\rm C^*}$-algebras in \cite{WW3}.

\begin{definition}(\cite{WW3}.)\label{def:2.2} Let $A$ be a ${\rm C^*}$-algebra, $m\in {\mathbb{N}}$.
 A completely  positive contraction  $\varphi:F\to A$ is  $m$-decomposable (where $F$ is a
finite dimensional  ${\rm C^*}$-algebra), if there is  a decomposition $F=F^{(0)}\oplus F^{(1)}
\oplus\cdots \oplus F^{(m)}$ such that the restriction $\varphi^{(i)}$ of $\varphi$ to $F^{(i)}$
has order zero (which means preserves orthogonality, i.e.,
 $\psi(e)\psi(f)=0$ for all $e, f\in {\rm M}_n$ with $ef=0$), for each $i\in \{0,\cdots,$$ m\}$,  and we say $\varphi$ is
$m$-decomposable with respect to  the decomposition $F=F^{(0)}\oplus F^{(1)}
\oplus\cdots \oplus F^{(m)}$.
 A has nuclear dimension $m$,  written ${\rm dim_{nuc}}(A)=m$, if $m$ is the least integer such that
the following condition holds: For any finite subset $G\subseteq A$ and $\varepsilon>0$, there is
a finite-dimensional completely  positive   approximation $(F,\varphi, \psi)$
for $G$ to within $\varepsilon$ (i.e., $F$ is finite-dimensional, $\psi: A\to F$  and
$\varphi:F\to A$ are completely positive,  and $\|\varphi\psi(b)-b\|<\varepsilon$
for any $b\in G$) such that $\psi$ is a   contraction, and $\varphi$ is $m$-decomposable with
 completely positive contraction order zero components $\varphi^{(i)}$. If no such $m$ exists, we write
${\rm dim_{nuc}}(A)=\infty$.
 \end{definition}

Hirshberg and  Orovitz  introduced  the notion of  tracial $\mathcal{Z}$-absorption  in \cite{HO}.

\begin{definition}(\cite{HO}.)\label{def:2.3} We say a unital ${\rm C^*}$-algebra $A$ is tracially $\mathcal{Z}$-absorbing if $A\neq {{\mathbb{C}}}$,
and for any finite set $F\subseteq A,$ $\varepsilon>0,$   non-zero positive element $a\in A$,
and $n\in {{\mathbb{N}}},$ there is  a completely positive order zero  contraction $\psi: {\rm M}_n\to A$, where  order zero  means preserving  orthogonality, i.e.,
 $\psi(e)\psi(f)=0$ for all $e,~f\in {\rm M}_n$ with $ef=0$, such that
the following properties hold:

$(1)$ $1-\psi(1)\precsim a,$ and

$(2)$ for any normalized element $x\in {\rm M}_n$ (i.e., with $\|x\|=1$)  and any $y\in F$ we have $\|\psi(x)y-y\psi(x)\|<\varepsilon$.
\end{definition}

Note that this  property implies that either $A=0$ or $\rm {dim}(A)=\infty$.

Inspired by  Hirshberg and  Orovitz's  tracial $\mathcal{Z}$-absorption in \cite{HO}, Fu introduced a  notion of  tracial nuclear dimension in his doctoral dissertation  \cite{FU} (see also  \cite{FL}), and he showed  that finite tracial nuclear dimension implies  tracial $\mathcal{Z}$-absorption for a separable, exact, simple  unital ${\rm C^*}$-algebra with non-empty tracial state space.

\begin{definition}(\cite{FU}.) \label{def:2.4} A unital ${\rm C^*}$-algebra $A$  is said to  have tracial nuclear dimension at most  $m$, written  ${\rm Trdim_{nuc}}(A)\leq m$, if for any
$\varepsilon>0$, any finite subset ${F}\subseteq A$, and any
non-zero positive element $a$ of $A$, there exist a  ${\rm C^*}$-subalgebra $D$ of $A$ with   ${\rm dim_{nuc}}(D) \leq m$ (Definition \ref {def:2.2}),  a contractive completely positive linear map $\varphi:A\to A$ and a contractive completely positive linear map $\psi: A\to D$
 such that

$(1)$ $\varphi(1)\precsim a$, and

$(2)$ $\|x-\varphi(x)-\psi(x)\|<\varepsilon$, for any $x\in F$.

\end{definition}

Let $\Omega$ be a class of unital ${\rm C^*}$-algebras. Then the class
 of simple  separable  ${\rm C^*}$-algebras which can be tracially
approximated by ${\rm C^*}$-algebras in $\Omega$, denoted by  ${\rm TA}\Omega$,  is
defined as follows.

\begin{definition}(\cite{EZ}.)\label{def:2.5} A  simple unital ${\rm C^*}$-algebra $A$ is said to      belong to the class ${\rm TA}\Omega$ if,   for any
 $\ep>0,$ any finite
subset $F\subseteq A,$ and any non-zero element $a\geq 0,$  there
are a  projection $p\in A$, and a ${\rm C^*}$-subalgebra $B$ of $A$ with
$1_B=p$ and $B\in \Omega$, such that

$(1)$ $\|xp-px\|<\ep$, for all $x\in  F$,

$(2)$ $pxp\in_\ep B$, for
all $x\in  F$, and

 $(3)$ $1-p\precsim a$.
\end{definition}

Remark: If $\Omega$ is a class of unital ${\rm C^*}$-algebras, by the proof of Theorem 4.1 of  \cite{EZ}, if $A$ is a  simple   unital  ${\rm C^*}$-algebra and  $A\in {\rm  TA}\Omega$ (Definition \ref{def:2.5}), then $A\in {\rm WTA}\Omega$  (Definition \ref {def:1.1}).
 If $\Omega$ is a class of  unital ${\rm C^*}$-algebras then the class  ${\rm TA}\Omega$ is contained in  the class ${\rm TAA}\Omega$
of Definition \ref {def:1.3}. (In particular one has Corollaries
\ref {cor:3.2},  \ref {cor:3.5}, and  \ref {cor:3.8}, below.)

Centrally large and stably centrally  large subalgebras were introduced in  \cite{AN} by Archey and Phillips.

\begin{definition}(\cite{AN}.)\label{def:2.6}
Let $A$ be an infinite-dimensional simple unital ${\rm C^*}$-algebra.
A unital ${\rm C^*}$-subalgebra $B \subseteq A$ is said to be centrally
large in $A$ if
for every $m \in {\mathbb{N}}$,
$a_1, a_2, \ldots, a_m \in A$,
$\ep > 0$, $x \in A_{+}$ with $\| x \| = 1$,
and $y \in B_{+}\setminus \{ 0 \}$,
there are $c_1, c_2, \ldots, c_m \in A$ and $g \in B$
such that the following conditions hold.

$(1)$ $0 \leq g \leq 1$.

$(2)$ For $j = 1, 2, \ldots, m$ we have
$\| c_j - a_j \| < \varepsilon$.

$(3)$ For $j = 1, 2, \ldots, m$ we have
$(1 - g) c_j \in B$.

$(4)$ $g \precsim_B y$ and $g \precsim_A x$.

$(5)$ $\| (1 - g) x (1 - g) \| > 1 - \varepsilon$.

$(6)$ For $j = 1, 2, \ldots, m$ we have
$\| g a_j - a_j g \| < \varepsilon$.
\end{definition}

Recall from Section $1$ that if a simple unital ${\rm C^*}$-algebra $A$ has a centrally
large  ${\rm C^*}$-subalgebra $B$, then $A$ belongs to
the class ${\rm WTA}\Omega$ with $\Omega=\{B\}$.
(In particular one has Corollaries
 \ref {cor:3.3}, \ref {cor:3.6}, and  \ref {cor:3.9}, below.)

The property  of   $m$-almost divisibility  was  introduced by Robert and Tikuisis in \cite{RT}.

\begin{definition}(\cite{RT}.)\label{def:2.7} Let $m\in \mathbb{N}$. We say that $A$ is  $m$-almost divisible if for each $a\in {\rm M}_{\infty}(A)_+,$ $k\in \mathbb{N}$,  and $\varepsilon>0,$ there exists $ b\in{\rm M}_{\infty}(A)_+$ such that $k\langle b\rangle\leq \langle a\rangle$ and $\langle (a-\varepsilon)_+\rangle\leq (k+1)(m+1)\langle b\rangle$.
\end{definition}

\begin{definition}\label{def:2.8} Let $n,~ m\in \mathbb{N}$. We shall say that $A$ is weakly   ($n, m$)-almost divisible if for each $a\in {\rm M}_{\infty}(A)_+,$ $k\in \mathbb{N}$,  and $\varepsilon>0,$ there exists $ b\in{\rm M}_{\infty}(A)_+$ such that $k\langle b\rangle\leq n\langle a\rangle$ and $\langle (a-\varepsilon)_+\rangle\leq (k+1)(m+1)\langle b\rangle$.
\end{definition}

Note that if $A$ has the property of either Definition \ref {def:2.7} or
Definition \ref {def:2.8} then so also does any matrix algebra over $A$.

The following two theorems are Lemma 1.7 and Lemma 1.8 of \cite{AJN}.

Theorem \ref{thm:2.10} will be used in the proof of Theorem \ref {thm:3.4}.

\begin{theorem}(\cite{AJN}.)\label{thm:2.9}
For every $\varepsilon>0$ there is $\delta>0$ such that the following statement holds. Let $A$ be a
${\rm C^*}$-algebra,  $B\subseteq A$  a ${\rm C^*}$-subalgebra,
 $n$  a non-zero integer, $\varphi_0:{\rm M}_n\to A$ a  completely positive contractive order zero map, and  $x\in B$ such that

$(1)$ $0\leq x\leq 1$,

$(2)$ with $(e_{j,k}), j,k=1,2,\cdots, n$  the standard system of matrix units for ${\rm M}_n$, we have
$\|\varphi_0(e_{j,k})x-x\varphi_0(e_{j,k})\|<\varepsilon$ for  $j,k=1,2,\cdots, n$, and

$(3)$ $\varphi_0(e_{j,k})x\in_{\varepsilon}B$.

Then there is a completely positive contractive order zero map $\varphi: {\rm M}_n\to B$ such that for all
$z\in {\rm M}_n$, with $\|z\|\leq 1$, we have $\|\varphi_0(z)x-\varphi(z)\|<\varepsilon$.
\end{theorem}

\begin{theorem}(\cite{AJN}.)\label{thm:2.10}
For every $\varepsilon>0$ and non-zero positive integer $n$,  there is $\delta>0$ such that the following statement holds.
Whenever $A,B$, $\varphi_0:{\rm M}_n\to A$, and $x\in B$ satisfy the conditions of  Theorem \ref{thm:2.9}, and in addition
$A$ is unital and $B$  contains  the unit of $A$, there exists a completely positive contractive order zero map  $\varphi:{\rm M}_n\to A$  such that

$(1)$ $\|\varphi_0(z)x-\varphi(z)\|<\varepsilon$, for all $z\in {\rm M}_n$ with $\|z\|\leq 1$, and

$(2)$ $1-\varphi(1)\precsim (1-x)\oplus (1-\varphi_0(1))$.

\end{theorem}
\section{The main results}
\begin{theorem}\label{thm:3.1}
 Let $\Omega$ be a class of unital
${\rm C^*}$-algebras which have  the property $\rm SP$. Then $A$ has the property $\rm SP$ for  any   simple  unital ${\rm C^*}$-algebra $A\in {\rm WTA}\Omega$.
\end{theorem}
\begin{proof}
Let $B$ be  a non-zero hereditary ${\rm C^*}$-subalgebra  of $A$. We must show
that  $B$   contains a non-zero projection.  Choose a positive element $a$ of $B$
of norm one.

Given   ${\varepsilon}>0$, with $f_{\varepsilon}$ as above,  there exists $\delta_2>0$  satisfying   Lemma 2.5.11 (2) of \cite{L2}.

With  $F=\{a\},$ and any $\varepsilon'>0,$
since
$A\in{\rm WTA}\Omega$,  there
exist a  projection $p\in A$,  an element $g\in A$ with $0\leq g\leq 1$,  and a $\rm C^*$-subalgebra $D$ of $A$ with $g\in D$ and
$1_D=p$, such that  $D$ has the  property  $\rm SP$  and

$(1)$ $(p-g)a\in_{\varepsilon'} D,~ a(p-g)\in_{\varepsilon'} D$,

$(2)$ $\|(p-g)a-a(p-g)\|<\varepsilon'$, and

 $(3)$ $\|(p-g)a(p-g)\|\geq 1-\varepsilon'$.

By $(1)$ and $(2)$, for sufficiently small $\varepsilon'$ (see Lemma 2.5.11 (2) of \cite{L2}), there exists  an element of norm at most one  $b\in D_+$
 such that $\|a{^{1/2}}(p-g)^2a{^{1/2}}-b\|<\delta_2$ and  $\|(p-g)a(p-g)-b\|<\delta_2$.

 Since $\|(p-g)a(p-g)-b\|<\delta_2$ and (by $(3)$) $\|(p-g)a(p-g)\|\geq 1-\varepsilon'$, one has $(b-\varepsilon)_+ \neq 0$ (otherwise, $1-\varepsilon'<\delta_2+\varepsilon$). Since $D$ has  the property $\rm SP$, then there exists a  non-zero projection $q\in \overline{(b-\varepsilon)_+D(b-\varepsilon)_+}$.

Since $f_{\varepsilon}(b)(b-\varepsilon)_+=(b-\varepsilon)_+$, we have $f_{\varepsilon}(b)q=q$.

Since $\|a{^{1/2}}(p-g)^2a{^{1/2}}-b\|<\delta_2$,  by the choice of $\delta_2$,
$$\|f_{\varepsilon}(a{^{1/2}}(p-g)^2a{^{1/2}})-f_{\varepsilon}(b)\|<{\varepsilon}.$$
Hence,
$$\begin{array}{ll}
&\|f_{\varepsilon}(a{^{1/2}}(p-g)^2a{^{1/2}})qf_{\varepsilon}(a{^{1/2}}(p-g)^2a{^{1/2}})-q\|\\
&= \|f_{\varepsilon}(a{^{1/2}}(p-g)^2a{^{1/2}})qf_{\varepsilon}(a{^{1/2}}(p-g)^2a{^{1/2}})-
f_{\varepsilon}(b)qf_{\varepsilon}(b)\|\\
 &<3{\varepsilon}.
\end{array}$$

It follows by the functional calculus that, when ${\varepsilon}$ is  small enough,
there exists a non-zero projection $e$ belonging to the hereditary
${\rm C^*}$-subalgebra of $A$ generated by $a$, and since $a\in B$ and $B$ is
hereditary, $e\in B$.
 This shows that $A$ has the property $\rm SP$.
\end{proof}

\begin{corollary}\label{cor:3.2}
 Let $\Omega$ be a class of  unital
${\rm C^*}$-algebras which have the property $\rm SP$. Then $A$ has  the property $\rm SP$ for  any  simple  unital ${\rm C^*}$-algebra $A\in {\rm TA}\Omega$.
\end{corollary}
\begin{proof} As pointed out in Section $1$, $\rm {TA}\Omega\subseteq \rm {WTA}\Omega$. The
 statement then follows from Theorem \ref{thm:3.1}.
\end{proof}

\begin{corollary}\label{cor:3.3}
 Let $A$ be a non-zero   simple unital ${\rm C^*}$-algebra, and let $B$ be a centrally large
 subalgebra of $A$. If $B$ has the  property $\rm SP$, then $A$ has the property $\rm SP$.
 \end{corollary}

\begin{proof} See remark following Definition \ref{def:2.6}.
\end{proof}
\begin{theorem}\label{thm:3.4}Let $\Omega$ be a class of   unital
${\rm C^*}$-algebras which  are tracially $\mathcal{Z}$-absorbing (Definition \ref {def:2.3}).  Then  $A$  is  tracially $\mathcal{Z}$-absorbing, if $A\neq {\mathbb{\mathbb{C}}}$,  for  any   simple unital  ${\rm C^*}$-algebra $A\in{\rm  WTA}\Omega.$\end{theorem}

\begin{proof}
We must show that
for any finite set $F=\{a_1,~a_2,~\cdots,~ a_k\}\subseteq A$ (we may assume that $\|a_i\|<1$ for all $1\leq i\leq k$),  any $\varepsilon>0$,  any non-zero positive element $b \in A$,
and any $n\in {{{\mathbb{N}}}},$  there is an order zero contraction $\psi:{\rm M}_n\to A$ such that
the following conditions hold:

$(1)$ $1-\psi(1)\precsim b$, and

$(2)$ for any normalized element $z\in {\rm M}_n$ and any $y\in F,$ we have
$\|\psi(z)y-y\psi(z)\|<\varepsilon.$

Since $A$ is either zero or  infinite-dimensional (see the remark following  Definition \ref{def:2.3}), and is simple, if $A\neq 0$ then by Lemma 2.3 of \cite{P3}, there exist  elements $b', ~b''\in A$  of norm one such that $b'b''=0$,
and $b'+b''\precsim b$. Also there exist elements  ${b_1}',~{b_2}'\in A$  of norm one such that ${b_1'}{b_2'}=0$, $b_1'\sim b_2'$,
and ${b_1'}+{b_2}'\precsim b'$.

 Given $\varepsilon>0$, with $f(t)=t^{1/2}\in C([0,1])$, there exists $\varepsilon'>0$ satisfying  Lemma 2.5.11 (1) of \cite{L2}. Given such $\varepsilon'>0$,  for $G=F\cup\{b'', ~ (b'')^{1/2}\}$,  since $A\in{\rm WTA}\Omega$,  there
exist a  projection $p\in A$, an  element $g\in A$ with $0\leq g\leq 1$,   and a  tracially $\mathcal{Z}$-absorbing ${\rm C}^*$-subalgebra $B$ of $A$ with $g\in B$ and
$1_B=p$  such that

$(1)'$ $(p-g)x\in_{\varepsilon'} B$, $x(p-g)\in_{\varepsilon'} B$, for $x\in G$,

$(2)'$ $\|(p-g)x-x(p-g)\|<\varepsilon'$, for $x\in G$,

$(3)'$ $1-(p-g)\precsim {b_1}'\sim{b_2}'$, and

$(4)'$ $\|(p-g)b''(p-g)\|\geq 1-\varepsilon'$.

 By  $(2)'$ with sufficiently small $\varepsilon'$, by  Lemma 2.5.11 (1) of \cite{L2}, we have

$(5)'$ $\|(p-g)^{1/2}x-x(p-g)^{1/2}\|<\varepsilon,$  for   $x\in G$, and

$(6)'$ $\|(1-(p-g))^{1/2}x-x(1-(p-g))^{1/2}\|<\varepsilon,$ for   $x\in G$.

By $(1)'$, with sufficiently small $\varepsilon'$, together with $(5)'$, there exist  elements $ a_1',~a_2', $ $~\cdots, ~a_k'\in B$ and a positive element $b'''\in B$ such that
 $$\|(p-g)^{1/2}a_i(p-g)^{1/2}-a_i'\|<\varepsilon, ~~~~ {\rm for}~  1\leq i\leq k, \rm {and}$$ $$\|(p-g)^{1/2}b''(p-g)^{1/2}-b'''\|<\varepsilon.$$

From the first inequality, together with $(5)'$ and $(6)'$,  for any $1\leq i\leq k$,  one has
$$\begin{array}{ll}
&\|a_i-a_i'-(1-(p-g))^{1/2}a_i(1-(p-g))^{1/2}\|\\
&\leq\|a_i-(p-g)a_i-(1-(p-g))a_i\|+\|(p-g)a_i-(p-g)^{1/2}a_i(p-g)^{1/2}\|\\
 &+\|(1-(p-g))a_i-(1-(p-g))^{1/2}a_i(1-(p-g))^{1/2}\|\\
 &+\|(p-g)^{1/2}a_i(p-g)^{1/2}-a_i'\|\\
  &<\varepsilon+\varepsilon+\varepsilon=3\varepsilon \hspace{0.8cm}(\textbf{3.4.1}).
\end{array}$$

From the second inequality,  by $(1)$ of  Theorem \ref{thm:2.1}, one has

$(7)'$ ~~$(b'''-\varepsilon)_+\precsim (p-g)^{1/2}b''(p-g)^{1/2}.$

By $(4)'$, if $\varepsilon'\leq \varepsilon$, then
$$\|(p-g)^{1/2}b''(p-g)^{1/2}\|\geq\|(p-g)b''(p-g)\|\geq 1-\varepsilon.$$
Hence by the second inequality again,
$$1-\varepsilon\leq\|(p-g)^{1/2}b''(p-g)^{1/2}\|\leq \|b'''\|+\|(p-g)^{1/2}b''(p-g)^{1/2}-b'''\|$$
$$\leq \|b'''\|+\varepsilon.$$ Therefore,
 $$\|(b'''-\varepsilon)_+\|\geq\|b'''\|-\varepsilon\geq 1-3\varepsilon.$$
 So, if $\varepsilon<1/3$, then  $\|(b'''-\varepsilon)_+\|>0.$

 Since $B\in \Omega,$  for  $H=\{a_1',~a_2',~\cdots, ~a_k',~ p-g,~ (p-g)^{1/2},~ (p-g)a_i'\}\subseteq B,$
  $\varepsilon>0$, $(b'''-\varepsilon)_+>0$, and $n$, there  is  an order zero contraction $\psi_0:{\rm M}_n\to B$  with
the following properties:

$(1)''$ $p-\psi_0(1)\precsim (b'''-\varepsilon)_+,$ and

$(2)''$ for any element $z\in {\rm M}_n$ of norm one,  and any $x\in H,$, we have
 $\|\psi_0(z)x-x\psi_0(z)\|<\varepsilon.$

 By Theorem \ref {thm:2.10}, applied with both the $A$ and $B$ of 2.10 equal to the present $B$,  there exists  a completely positive contractive order zero map $\psi:{\rm M}_n\to B$ such that

$(1)'''$ $\|\psi(z)-\psi_0(z)(p-g)\|<\varepsilon$, and

$(2)'''$  $p-\psi(1)\precsim (p-(p-g))\oplus (p-\psi_0(1))$.

We then have
$$\begin{array}{ll}
&1-\psi(1)=1-p+p-\psi(1)\precsim (1+g-p)\oplus(p-\psi(1))\\
&\precsim {b_1}'\oplus (p-(p-g))\oplus(p-\psi_0(1)) ~~( {\rm by} ~~(3)'  {\rm and}~~ (2)''')\\
 &\precsim {b_1}'\oplus(1-(p-g))\oplus (b'''-\varepsilon)_+ ~~( {\rm by} ~~(1)'')\\
 &\precsim {b_1}'\oplus {b_2}'\oplus (p-g)^{1/2}b''(p-g)^{1/2} ~({\rm by } ~~(3)' {\rm and} ~~~(7)')\\
  &\precsim b'+b''\precsim b.
\end{array}$$
This is $(1)$ above.

For any element $z\in {\rm M}_n$ of norm one, any $a_i'\in F$,
we have (by $(1)'''$, $(2)''$, the choice of $a_i'$, $(2)'$ for small enough $\varepsilon'$, the choice of  $a_i'$, and  $(1)'''$)

$$\begin{array}{ll}
&\|\psi(z)a_i'-a_i'\psi(z)\|\\
&\leq\|\psi(z)a_i'-\psi_0(z)(p-g)a_i'\|
+\|\psi_0(z)(p-g)a_i'-(p-g)a_i'\psi_0(z)\|\\
 &+\|(p-g)a_i'\psi_0(z)-(p-g)(p-g)^{1/2}a_i(p-g)^{1/2}\psi_0(z)\|\\
 &+\|(p-g)(p-g)^{1/2}a_i(p-g)^{1/2}\psi_0(z)-(p-g)^{1/2}a_i(p-g)^{1/2}(p-g)\psi_0(z)\|\\
 &+\|(p-g)^{1/2}a_i(p-g)^{1/2}(p-g)\psi_0(z)-a_i'(p-g)\psi_0(z)\|\\
 &+\|a_i'(p-g)\psi_0(z)-a_i'\psi(z)\|<\varepsilon+\varepsilon+\varepsilon+\varepsilon+
 \varepsilon+\varepsilon=6\varepsilon \hspace{0.8cm}(\textbf{3.4.2}).
\end{array}$$

We also have (by $(1)'''$, $(6)'$,  $(2)''$, $(2)'$ with $\varepsilon'\leq \varepsilon$, $(5)'$,  the choice of $a_i'$, $(2)''$, the choice of $a_i'$,
 $(5)'$, $(1)'''$, and  $(6)'$)
$$\begin{array}{ll}
&\|\psi(z)(1-(p-g))^{1/2}a_i(1-(p-g))^{1/2}-(1-(p-g))^{1/2}a_i(1-(p-g))^{1/2}\psi(z)\|\\
&\leq\|\psi(z)(1-(p-g))^{1/2}a_i((1-(p-g))^{1/2}\\
 &-\psi_0(z)(p-g)(1-(p-g))^{1/2}a_i(1-(p-g))^{1/2}\|\\
 &+\|\psi_0(z)(p-g)(1-(p-g))^{1/2}a_i(1-(p-g))^{1/2}-\psi_0(z)(p-g)(1-(p-g))a_i\|\\
 &+\|\psi_0(z)(1-(p-g))(p-g)a_i-(1-(p-g))\psi_0(z)(p-g)a_i\|\\
 &+\|(1-(p-g))\psi_0(z)(p-g)a_i-(1-(p-g))\psi_0(z)a_i(p-g)\|\\
 &+\|(1-(p-g))\psi_0(z)a_i(p-g)-(1-(p-g))\psi_0(z)(p-g)^{1/2}a_i(p-g)^{1/2}\|\\
 &+\|(1-(p-g))\psi_0(z)(p-g)^{1/2}a_i(p-g)^{1/2}-(1-(p-g))\psi_0(z)a_i'\|\\
 &+\|(1-(p-g))\psi_0(z)a_i'-(1-(p-g))a_i'\psi_0(z)\|\\
 &+\|(1-(p-g))a_i'\psi_0(z)-(1-(p-g))(p-g)^{1/2}a_i(p-g)^{1/2}\psi_0(z)\|\\
 &+\|(1-(p-g))(p-g)^{1/2}a_i(p-g)^{1/2}\psi_0(z)-(1-(p-g))a_i(p-g)\psi_0(z)\|\\
 &+\|(1-(p-g))a_i(p-g)\psi_0(z)-(1-(p-g))a_i\psi(z)\|\\
 &+\|(1-(p-g))a_i\psi(z)-(1-(p-g))^{1/2}a_i(1-(p-g))^{1/2}\psi(z)\|\\
 &<\varepsilon+\varepsilon+\varepsilon+\varepsilon+\varepsilon+\varepsilon+\varepsilon+\varepsilon+\varepsilon+
 \varepsilon+\varepsilon=11\varepsilon  \hspace{0.8cm}(\textbf{3.4.3}).
\end{array}$$
Therefore, for any $a_i\in F$,  we have (by $(\textbf{3.4.1})$,  $(\textbf{3.4.1})$, $(\textbf{3.4.2})$, and $(\textbf{3.4.3})$)
$$\begin{array}{ll}
&\|\psi(z)a_i-a_i\psi(z)\|\\
&\leq\|\psi(z)a_i-\psi(z)(a_i'+(1-(p-g))^{1/2}a_i(1-(p-g))^{1/2})\|\\
 &+\|\psi(z)(a_i'+(1-(p-g))^{1/2}a_i(1-(p-g))^{1/2})\\
 &-(a_i'+(1-(p-g))^{1/2}a_i(1-(p-g))^{1/2})\psi(z)\|\\
 &+\|(a_i'+(1-(p-g))^{1/2}a_i(1-(p-g))^{1/2})\psi(z)-a_i\psi(z)\|\\
 &\leq 3\varepsilon+3\varepsilon+\|\psi(z)a_i'-a_i'\psi(z)\|\\
 &+ \|\psi(z)((1-(p-g))^{1/2}a_i((1-(p-g))^{1/2}\\
 &-((1-(p-g))^{1/2}a_i((1-(p-g))^{1/2}\psi(z)\|\\
 &\leq 6\varepsilon+6\varepsilon+11\varepsilon=23\varepsilon.
\end{array}$$
This is $(2)$ above, with $23\varepsilon$ in place of  $\varepsilon$.
\end{proof}

The following  corollary  was obtained by Elliott, Fan, and Fang in \cite{EFF}.

\begin{corollary}\label{cor:3.5}(\cite{EFF}.)
 Let $\Omega$ be a class of  unital
${\rm C^*}$-algebras which are tracially $\mathcal{Z}$-absorbing. Then $A$ is tracially $\mathcal{Z}$-absorbing, or else $A=\mathbb{C}$,  for  any  simple   unital ${\rm C^*}$-algebra $A\in {\rm TA}\Omega$.
\end{corollary}
\begin{proof}
  As pointed out in Section $1$, $\rm {TA}\Omega\subseteq \rm {WTA}\Omega$. The
 statement then follows from  Theorem \ref{thm:3.4}.
\end{proof}

The following  corollary  was obtained by   Archey, Buck, and  Phillips in \cite{AJN}.

\begin{corollary}\label{cor:3.6}(\cite{AJN}.)
 Let $A$ be a simple  unital ${\rm C^*}$-algebra, and let $B$ be a centrally large
 subalgebra of $A$. If $B$ is tracially $\mathcal{Z}$-absorbing, then $A$ is tracially $\mathcal{Z}$-absorbing.
 \end{corollary}
\begin{proof}
  See remark following Definition \ref{def:2.6}.
\end{proof}
\begin{theorem}\label{thm:3.7}Let $\Omega$ be a class of unital nuclear ${\rm C^*}$-algebras with tracial nuclear dimension at most $n$ (Definition \ref {def:2.4}).   Then  $A$  has tracial nuclear dimension at most $n$ for  any  simple  unital ${\rm C^*}$-algebra $A\in{\rm  WTA}\Omega.$\end{theorem}

\begin{proof}
Let $A$ be a simple unital ${\rm C^*}$-algebra in $\rm{WTA}\Omega$.  We must show that for any
$\varepsilon>0$, any finite subset $F=\{a_1,a_2, $ $\cdots, a_n\}$ of $A$,  and any
non-zero positive element $b$ of $A$,  there exist  a ${\rm C^*}$-subalgebra  $D$ of $A$ with ${\rm dim_{nuc}}(D)\leq m$, a contractive completely positive linear map $\varphi:A\to A$, and a contractive completely positive linear map $\psi: A\to D$ such that

$(1)$ $\varphi(1)\precsim b$, and

$(2)$ $\|x-\varphi(x)-\psi(x)\|<\varepsilon$, for any $x\in F$.

We shall show this with $8\varepsilon$ in place of $\varepsilon$.
By Lemma 2.3 of \cite{P3},   there exist  positive elements $b_1,~b_2\in A$ of norm one such that
$b_1b_2=0$,  $b_1\sim b_2$, and $b_1+b_2\precsim b$.

Given $\varepsilon>0$, with $f(t)=t^{1/2}\in C([0,1])$, there exists $\varepsilon'>0$ satisfying  Lemma 2.5.11 (1) of \cite{L2}. Given such $\varepsilon'>0$, for $G=F\cup\{b_2\}$,
  since   $A\in{\rm  WTA}\Omega$  there
exist a  projection $p\in A$,   an element $g\in A$ with $0\leq g\leq 1$,   and  a $\rm C^*$-subalgebra $B$ of $A$ with $g\in B$,
$1_B=p$,  and  ${\rm Trdim}_{nuc}(B)\leq m$ such that

$(1)'$ $(p-g)x\in_{\varepsilon'} B,~ x(p-g)\in_{\varepsilon'} B$,  for $x\in G$,

$(2)'$ $\|(p-g)x-x(p-g)\|<\varepsilon'$,  for $x\in G$,

$(3)'$ $1-(p-g)\precsim b_1 \sim{b_2}$, and

$(4)'$ $\|(p-g)b_2(p-g)\|>1-\varepsilon'$.

By  $(2)'$ and Lemma 2.5.11  (1) of \cite{L2}, if $\varepsilon'$ is sufficiently small, we have

$(5)'$ $\|(p-g)^{1/2}x-x(p-g)^{1/2}\|<\varepsilon,$ for any $x\in G$,

$(6)'$ $\|(1-(p-g))^{1/2}x-x(1-(p-g))^{1/2}\|<\varepsilon,$
for any $x\in G$.

By $(1)'$, with sufficiently small $\varepsilon'$, together with $(5)'$, there exist elements $a_1',~ a_2',~ \cdots, ~ a_n'\in B$ and a positive element $b_2'\in B$ such that  $$\|(p-g)^{1/2}a_i(p-g)^{1/2}-a_i'\|<\varepsilon,  ~~~ \rm{for} ~1\leq i\leq n,~~ \rm {and} $$  $$\|(p-g)^{1/2}b_2(p-g)^{1/2}-b_2'\|<\varepsilon.$$

From the first inequality, together with $(5)'$ and $(6)'$, for any $1\leq i\leq n$,  one has
$$\begin{array}{ll}
&\|a_i-a_i'-(1-(p-g))^{1/2}a_i(1-(p-g))^{1/2}\|\\
&\leq\|a_i-(p-g)a_i-(1-(p-g))a_i\|+\|(p-g)a_i-(p-g)^{1/2}a_i(p-g)^{1/2}\|\\
 &+\|(1-(p-g))a_i-(1-(p-g))^{1/2}a_i(1-(p-g))^{1/2}\|\\
 &+\|(p-g)^{1/2}a_i(p-g)^{1/2}-a_i'\|\\
  &<\varepsilon+\varepsilon+\varepsilon=3\varepsilon \hspace{0.8cm}(\textbf{3.7.1}).
\end{array}$$

Since  $\|(p-g)^{1/2}b_2(p-g)^{1/2}-b_2'\|<\varepsilon $, by $(1)$ of Theorem \ref{thm:2.1}, we have

$(7)'$ $(b_2'-3\varepsilon)_+\precsim ((p-g)^{1/2}b_2(p-g)^{1/2}-2\varepsilon)_+ .$

By  $(4)'$, with $\varepsilon'\leq \varepsilon<1/5$, $$\|(p-g)^{1/2}b_2(p-g)^{1/2}\|\geq\|(p-g)b_2(p-g)\|\geq 1-\varepsilon.$$

Therefore,  by the choice of $b_2'$, one has  $$\|(b_2'-3\varepsilon)_+\|\geq \|(p-g)^{1/2}b_2(p-g)^{1/2}\|-4\varepsilon\geq 1-5\varepsilon.$$  In particular,  $(b_2'-3\varepsilon)_+\neq 0$.

Define  a contractive completely positive linear map $\varphi'': A\to A$ by $\varphi''(a)=(1-(p-g))^{1/2}a(1-(p-g))^{1/2}$.
Since $B$ is a nuclear ${\rm C^*}$-algebra, by Theorem 2.3.13 of  \cite{L2} there exists a contractive completely positive linear map
  $\psi'': A\to B$ such that $\|\psi''(p-g)-(p-g)\|<\varepsilon$ and $\|\psi''(a_i')-a_i'\|<\varepsilon$ for all $1\leq i\leq n$.

Since ${\rm Trdim_{nuc}}(B)\leq m$,  there exist  a contractive completely positive linear map $\varphi':B\to B$  and a contractive completely positive linear map $\psi': B\to D$
 with ${\rm dim_{nuc}}(D)\leq m$
 such that

$(1)''$ $\varphi'(p)\precsim (b_2'-3\varepsilon)_+$, and

$(2)''$ $\|(p-g)-\varphi'(p-g)-\psi'(p-g)\|<\varepsilon$, and $\|a_i'-\varphi'(a_i')-\psi'(a_i')\|<\varepsilon$, for all $1\leq i\leq n$.

Define $\bar{\varphi}:A\to A$ by $\bar{\varphi}(a)=\varphi''(a)+\varphi'(\psi''((p-g)^{1/2}a(p-g)^{1/2}))$, $\varphi=\frac{1}{1+2\varepsilon}\bar{\varphi}$,
and  $\psi: A\to D$ by
 $\psi(a)=\psi'(\psi''((p-g)^{1/2}a(p-g)^{1/2})))$.
 Then
 $$\begin{array}{ll}
&\|\varphi(1)\|=\frac{1}{1+2\varepsilon}\|\bar{\varphi}(1)\|\\
&=\frac{1}{1+2\varepsilon}\|\varphi''(1)+\varphi'(\psi''((p-g)^{1/2}1(p-g)^{1/2}))\|\\
 &=\frac{1}{1+2\varepsilon}\|1-(p-g)-\varphi'(\psi''(p-g))\|\\
 &=\frac{1}{1+2\varepsilon}\|1-(p-g)+\varphi'(p-g)+\psi'(p-g)+\varphi'(\psi''(p-g))
 -\varphi'(p-g)-\psi'(p-g)\|\\
 &\leq \frac{1}{1+2\varepsilon}(\|(p-g)-\varphi'(p-g)-\psi'(p-g)\|\\
 &+ \|\varphi'(\psi''(p-g))-\varphi'(p-g)\|+\|1-\psi'(p-g)\|)\\
 &\leq
 \frac{1}{1+2\varepsilon}(1+2\varepsilon)=1 ~~ (\rm {by} (2)''~~ \rm {and~~ definition~~ of} ~~~\psi'') .
\end{array}$$
 Therefore, $\varphi$ is a contractive completely positive linear map. Also $\psi$ is a contractive completely positive linear map.

 We have
$$\begin{array}{ll}
&\varphi(1)=\frac{1}{1+2\varepsilon}(\varphi''(1)+\varphi'(\psi''(p-g)))\\
&\sim \varphi''(1)+\varphi'(\psi''(p-g))\\
&\precsim 1-(p-g)\oplus \varphi'(p)\precsim b_1\oplus (b_2'-3\varepsilon)_+ ~~~({\rm by} ~~(3)' {\rm and} ~~(1)'')\\
 &\precsim b_1\oplus ((p-g)^{1/2}b_2(p-g)^{1/2}-2\varepsilon)_+ ~~~(\rm {by} ~~(7)')\\
 &\precsim b_1\oplus b_2\sim b_1+b_2\precsim b,
\end{array}$$
 and (by $(\textbf{3.7.1})$, $(2)''$, the choice of $\psi''$, the definition of $a_i'$, and the last two again)
 $$\begin{array}{ll}
&\|a_i-\varphi(a_i)-\psi(a_i)\|\\
&=\|a_i-\frac{1}{1+2\varepsilon}\varphi''(a_i)-\frac{1}{1+2\varepsilon}\varphi'(\psi''((p-g)^{1/2}a_i(p-g)^{1/2}))\\
&-\psi'(\psi''((p-g)^{1/2}a_i(p-g)^{1/2}))\|\\
 &\leq  \|a_i-\frac{1}{1+2\varepsilon}(1-(p-g))^{1/2}a_i(1-(p-g))^{1/2}-a_i'\|\\
 &+\|a_i'-\frac{1}{1+2\varepsilon}\varphi'(\psi''((p-g)^{1/2}a_i(p-g)^{1/2}))-\psi'(\psi''((p-g)^{1/2}a_i(p-g)^{1/2}))\|\\
 &\leq\|a_i-(1-(p-g))^{1/2}a_i(1-(p-g))^{1/2}-a_i'\|\\
 &+\|(1-(p-g))^{1/2}a_i(1-(p-g))^{1/2}-\frac{1}{1+2\varepsilon}(1-(p-g))^{1/2}a_i(1-(p-g))^{1/2}\|\\
 &+\|a_i'-\varphi'(a_i')-\psi'(a_i')\|\\
 &+\|\varphi'(a_i')-\varphi'(\psi''(a_i'))\|\\
 &+\|\varphi'(\psi''(a_i'))-\varphi'(\psi''((p-g)^{1/2}a_i(p-g)^{1/2}))\|\\
 &+\|\varphi'(\psi''((p-g)^{1/2}a_i(p-g)^{1/2}))-\frac{1}{1+2\varepsilon}\varphi'(\psi''((p-g)^{1/2}a_i(p-g)^{1/2}))\|\\
 &+\|\psi'(a_i')-\psi'(\psi''(a_i'))\|\\
 &+\|\psi'(\psi''(a_i'))-\psi'(\psi''((p-g)^{1/2}a_i(p-g)^{1/2}))\|\\
  &\leq 3\varepsilon+\frac{2\varepsilon}{1+2\varepsilon}+\varepsilon+\varepsilon+\varepsilon+\frac{2\varepsilon}{1+2\varepsilon}+\varepsilon+\varepsilon
  =8\varepsilon+\frac{4\varepsilon}{1+2\varepsilon}.
\end{array}$$

Thus we have $(1)$, and $(2)$ with $8\varepsilon+\frac{4\varepsilon}{1+2\varepsilon}$ in place of $\varepsilon$.
\end{proof}

The following  corollary  was obtained by  Fan and Yang in \cite{Q9}.

\begin{corollary}\label{cor:3.8}
 Let $\Omega$ be a class of  unital
${\rm C^*}$-algebras such that ${\rm Trdim_{nuc}}(B)$ $\leq m$ for any $B\in \Omega$. Then ${\rm Trdim_{nuc}}(A)$ $\leq m$  for  any  simple unital ${\rm C^*}$-algebra $A\in {\rm TA}\Omega$.
\end{corollary}
\begin{proof}
 As pointed out in Section $1$, $\rm {TA}\Omega\subseteq \rm {WTA}\Omega$. The
 statement then follows from  Theorem \ref{thm:3.7}.
\end{proof}

The following  corollary  was obtained by  Zhao, Fang,  and Fan in \cite{ZF}.
 \begin{corollary}\label{cor:3.9}
 Let $A$ be a simple  unital ${\rm C^*}$-algebra, and let $B$  be a nuclear centrally large
 subalgebra of $A$. If ${\rm Trdim_{nuc}}(B)$ $\leq m$, then ${\rm Trdim_{nuc}}(A)$ $\leq m$.
 \end{corollary}
\begin{proof}
  See remark following Definition \ref{def:2.6}.
\end{proof}

\begin{theorem}\label{thm:3.10}
 Let $\Omega$ be a class of   unital
${\rm C^*}$-algebras  which are  $m$-almost divisible (Definition \ref {def:2.7}). Let  $A\in{\rm WTA}\Omega$ be a  simple unital stably finite ${\rm C^*}$-algebra such that for any $n\in \mathbb{N}$
the  ${\rm C^*}$-algebra $\rm{M}$$_n(A)$  belongs to the class ${\rm WTA}\Omega$. Then
   $A$  is weakly $(2, m)$-almost divisible (Definition \ref {def:2.8}).
\end{theorem}

\begin{proof} We  must show that there is $b\in {\rm M}_{\infty}(A)_+$  such that $k\langle b\rangle \leq  \langle a \rangle+\langle a \rangle$   and $\langle (a-\varepsilon)_+\rangle \leq (k+1)(m+1)\langle b\rangle$
 for any  given $a\in A_+,$  $\varepsilon>0$, and $k\in \mathbb{N}.$
  We may assume that $\|a\|=1.$ (We have replaced $\rm {M}$$_n(A)$ containing $a$ given initially  by $A$.)

For any $\delta_1>0$, since   $A\in{\rm WTA}\Omega$,  there
exist a  projection $p\in A$, an  element $g\in A$ with $0\leq g\leq 1$,  and a $C^*$-subalgebra $B$ of $A$ with $g\in B$,
$1_B=p$, and $B\in \Omega$ such that

$(1)$  $(p-g)a\in_{\delta_1} B$, and

$(2)$ $\|(p-g)a-a(p-g)\|<{\delta_1}$.

By  $(2)$, with sufficiently small $\delta_1$,  by Lemma 2.5.11  $(1)$ of \cite{L2}, we have

$(3)$ $\|(p-g)^{1/2}a-a(p-g)^{1/2}\|<\varepsilon/3$,  and

$(4)$ $\|(1-(p-g))^{1/2}a-a(1-(p-g))^{1/2}\|<\varepsilon/3$.

By $(1)$ and $(2)$, with sufficiently small $\delta_1$, there exists a  positive element $a'\in B$  such that

 $(5)$ $\|(p-g)^{1/2}a(p-g)^{1/2}-a'\|<{\varepsilon/3}.$

By $(3)$, $(4)$,  and $(5)$,
$$\begin{array}{ll}
&\|a-a'-(1-(p-g))^{1/2}a(1-(p-g))^{1/2}\|\\
&\leq\|a-(p-g)a-(1-(p-g))a\|+\|(p-g)a-(p-g)^{1/2}a(p-g)^{1/2}\|\\
 &+\|(1-(p-g))a-(1-(p-g))^{1/2}a(1-(p-g))^{1/2}\|\\
 &+\|(p-g)^{1/2}a(p-g)^{1/2}-a'\|\\
  &<\varepsilon/3+\varepsilon/3+\varepsilon/3=\varepsilon.  \hspace{0.8cm}(\textbf{3.10.1})
\end{array}$$

 Since  $B$ is   $m$-almost divisible, and $(a'-3{\varepsilon})_+\in B$,  there exists  $b_1\in B$ such that $k\langle b_1\rangle\leq \langle (a'-3{\varepsilon})_+\rangle$
 and $ \langle (a'-4{\varepsilon})_+\rangle\leq (k+1)(m+1)\langle b_1\rangle$.

 Since  $B$ is  $m$-almost divisible, and $(a'-2{\varepsilon})_+\in B$,  there exists  $b'\in B$ such that $k\langle b'\rangle\leq \langle (a'-2{\varepsilon})_+\rangle$
 and $ \langle (a'-3{\varepsilon})_+\rangle\leq (k+1)(m+1)\langle b'\rangle$.

Write $a''=(1-(p-g))^{1/2}a(1-(p-g))^{1/2}$.

We  divide the proof into two cases.

\textbf{Case (1)} We assume that $(a'-3{\varepsilon})_+$ is Cuntz equivalent to a projection.

\textbf{(1.1)} We assume that $(a'-4{\varepsilon})_+$ is Cuntz equivalent to a projection.

\textbf{(1.1.1)} If  $\langle (a'-4{\varepsilon})_+\rangle$, the class of a projection,  is not equal to $(k+1)(m+1)\langle b_1\rangle$, by Theorem 2.1 (2),  there exists  non-zero $c\in A_+$ such that $\langle (a'-4{\varepsilon})_+\rangle+\langle c\rangle \leq (k+1)(m+1)\langle b_1\rangle$.

For any  $\delta_2>0$,  since   $A\in{\rm  WTA}\Omega$,  there
exist a  projection $p'\in A$,  an  element $g_1\in A$ with $0\leq g_1\leq 1$,  and a $C^*$-subalgebra $D$ of $A$ with  $g_1\in D$,
$1_D=p'$, and $D\in \Omega$ such that

$(1)'$  $(p'-g_1)a''\in_{\delta_2} D$,

$(2)'$ $\|(p'-g_1)a''-a''(p'-g_1)\|<\delta_2$, and

$(3)'$ $1-(p'-g_1)\precsim  c$.

 By $(1)'$ and $(2)'$,  with sufficiently small $\delta_2$, as above, via the analogues of $(4)$, $(5)$, and $(6)$ for $a'', p'$, and $g_1$, there exists a  positive element $a'''\in D$  such that
 $$\|(p'-g_1)^{1/2}a''(p'-g_1)^{1/2}-a'''\|<{\varepsilon/3}, ~~~ \rm {and}$$
  $$\|a''-a'''-(1-(p'-g_1))^{1/2}a''(1-(p'-g_1))^{1/2}\|<{\varepsilon}.  \hspace{0.8cm}(\textbf{3.10.2})$$

  Since  $D$ is   $m$-almost divisible, and $(a'''-3{\varepsilon})_+\in D$,  there exists  $b_2\in D_+$ such that $k\langle b_2\rangle \leq \langle (a'''-3{\varepsilon})_+\rangle$
 and $ \langle (a'''-4{\varepsilon})_+\rangle\leq (k+1)(m+1)\langle b_2\rangle$.

  Since $a'\leq a'+a''$, one has $\langle (a'-{\varepsilon})_+\rangle\leq \langle (a'+a''-{\varepsilon})_+\rangle$ (by Theorem 2.1 (5)).   By $(\textbf{3.10.1})$, $\|a-a'-a''\|<{\varepsilon}$ and hence  by Theorem 2.1 (1), one also has
 $\langle (a'+a''-{\varepsilon})_+\rangle\leq  \langle a\rangle$. Therefore, $\langle (a'-2{\varepsilon})_+\rangle \leq \langle (a'-{\varepsilon})_+\rangle\leq \langle a\rangle$.

 Similarly, with $x=(1-(p'-g_1))^{1/2}a''(1-(p'-g_1))^{1/2}$, $a'''\leq a'''+x$ implies
 $\langle (a'''-{3\varepsilon})_+\rangle \leq \langle(a'''+x-{3\varepsilon})_+ \rangle$ and $(\textbf{3.10.2})$, i.e.,
  $\|a''-a'''-x\|<{\varepsilon}$, implies $\langle(a'''+x-3{\varepsilon})_+ \rangle\leq \langle a'' \rangle \leq \langle a \rangle $.

 Therefore, we have
\begin{eqnarray}
\label{Eq:eq1}
  &&k\langle b_1\oplus b_2\rangle= k\langle b_1\rangle+ k\langle b_2\rangle\nonumber\\
  &&\leq \langle (a'-3{\varepsilon})_+\rangle +\langle (a'''-3{\varepsilon})_+\rangle \nonumber\\
  &&\leq \langle (a'-2{\varepsilon})_+\rangle +\langle (a'''-3{\varepsilon})_+\rangle \nonumber\\
  &&\leq \langle a\rangle+ \langle a\rangle.\nonumber
\end{eqnarray}

 By $(\textbf{3.10.1})$ and $(\textbf{3.10.2})$,
 $$\|a-a'-a'''-(1-(p'-g_1))^{1/2}a''(1-(p'-g_1))^{1/2}\|<2{\varepsilon},$$
  and therefore,  by Theorem 2.1 (1) and Theorem 2.1 (4) (twice),
\begin{eqnarray}
\label{Eq:eq1}
  &&\langle (a-10\varepsilon)_+\rangle\nonumber\\
  &&\leq \langle (a'-4{\varepsilon})_+\rangle + \langle (a'''-4{\varepsilon})_+ \rangle+\langle (1-(p'-g_1))^{1/2}a''(1-(p'-g_1))^{1/2}\rangle\nonumber\\
 &&\leq \langle (a'-4{\varepsilon})_+\rangle + \langle (a'''-4{\varepsilon})_+ \rangle+ \langle 1-(p'-g_1)\rangle\nonumber\\
&&\leq \langle (a'-4{\varepsilon})_+\rangle + \langle (a'''-4{\varepsilon})_+ \rangle+  \langle c  \rangle ~(\rm{ by} ~(3)')\nonumber\\
&&\leq (k+1)(m+1)\langle b_1 \rangle+ (k+1)(m+1)\langle b_2 \rangle=(k+1)(m+1)\langle b_1\oplus b_2\rangle.\nonumber
\end{eqnarray}

These are  the desired inequalities, with $b_1\oplus b_2$ in place of $b$ and  $10\varepsilon$ in place of $\varepsilon$.

\textbf{(1.1.2)} If $\langle  (a'-4{\varepsilon})_+ \rangle$  is  equal  to $(k+1)(m+1)\langle b_1 \rangle$,  so that   $(k+1)(m+1)\langle b_1 \rangle \leq \langle (a'-3{\varepsilon})_+ \rangle $, then, as $m+1\geq 2$, we have $k\langle b_1\oplus b_1 \rangle \leq \langle (a'-3{\varepsilon})_+ \rangle $. Also,
 $\langle (a'-4{\varepsilon})_+\rangle +\langle  b_1 \rangle \leq  (k+1)(m+1)\langle b_1\oplus b_1\rangle $.

As  in the part   \textbf{(1.1.1)},  as $A\in \rm {WTA}\Omega$ (with $b_1$ in place of $c$ in the part \textbf{(1.1.1)}), there
exist a  projection $p''\in A$, an  element $g_2\in A$ with $0\leq g_2\leq 1$,  and a $C^*$-subalgebra $D_1$ of $A$ with $g_3\in D_1$,
$1_{D_1}=p''$, and $D_1\in \Omega$, and there exists a  positive element $a^{(4)}\in D_1$ such that
 $$\|a''-a^{(4)}-(1-(p''-g_2))^{1/2}a''(1-(p''-g_2))^{1/2}\|<{\varepsilon},
 \hspace{0.8cm}(\textbf{3.10.3}) $$
  $$\langle (a^{(4)}-3{\varepsilon})_+\rangle\leq \langle  a \rangle,~~~ {\rm {and}}~~~\langle 1-(p''-g_2) \rangle\leq \langle b_1 \rangle.$$

 Since  $D_1$ is   $m$-almost divisible, and $(a^{(4)}-3{\varepsilon})_+\in D_1$,  there exists  $b_3\in (D_1)_+$ such that $k\langle b_3 \rangle \leq \langle (a^{(4)}-3{\varepsilon})_+ \rangle $
 and $\langle  (a^{(4)}-4{\varepsilon})_+ \rangle \leq  (k+1)(m+1)\langle b_3 \rangle$.

  Then, as in the part  \textbf{(1.1.1)},
\begin{eqnarray}
\label{Eq:eq1}
  &&k\langle  b_1\oplus b_1\oplus b_3 \rangle= k\langle b_1\oplus b_1 \rangle + k\langle b_3 \rangle \nonumber\\
  &&\leq \langle (a'-3{\varepsilon})_+ \rangle +\langle (a^{(4)}-3{\varepsilon})_+\rangle \nonumber\\
  &&\leq \langle (a'-2{\varepsilon})_+ \rangle +\langle (a^{(4)}-3{\varepsilon})_+\rangle \nonumber\\
&&\leq\langle  a \rangle+\langle  a \rangle.\nonumber
\end{eqnarray}

By $(\textbf{3.10.1})$ and $(\textbf{3.10.3})$,
 $$\|a-a'-a^{(4)}-(1-(p''-g_2))^{1/2}a''(1-(p''-g_2))^{1/2}\|<2{\varepsilon},$$
 and  therefore, as in the part \textbf{(1.1.1)}
\begin{eqnarray}
\label{Eq:eq1}
  &&\langle (a-10\varepsilon)_+ \rangle \nonumber\\
  &&\leq \langle (a'-4{\varepsilon})_+ \rangle+\langle (a^{(4)}-4{\varepsilon})_+ \rangle  +\langle (1-(p''-g_2))^{1/2}a''(1-(p''-g_2))^{1/2}\rangle \nonumber\\
  &&\leq \langle (a'-4{\varepsilon})_+ \rangle +\langle  (a^{(4)}-4{\varepsilon})_+ \rangle +\langle 1-(p''-g_2) \rangle \nonumber\\
&&\leq \langle (a'-4{{\varepsilon}})_+ \rangle+\langle b_1 \rangle+\langle (a^{(4)}-{4\varepsilon})_+ \rangle\nonumber\\
&&\leq (k+1)(m+1)\langle b_1\oplus b_1 \rangle +(k+1)(m+1)\langle b_3 \rangle\nonumber\\
&&= (k+1)(m+1)\langle b_1\oplus b_1\oplus b_3\rangle.\nonumber
\end{eqnarray}

These are  the desired inequalities, with $b_1\oplus b_1\oplus b_3$ in place of $b$  and  $10\varepsilon$ in place of $\varepsilon$.

\textbf{(1.2)} We assume that $(a'-4{\varepsilon})_+$ is not Cuntz equivalent to a projection.
 By Theorem 2.1 (3), there is a non-zero positive element $d$  such that
 $\langle (a'-5{\varepsilon})_+ \rangle +\langle d \rangle \leq \langle  (a'-4{\varepsilon})_+ \rangle$.

As in the part  \textbf{(1.1.1)}, as $A\in \rm {WTA}\Omega$ (with $d$ in place of $c$ in the part  \textbf{(1.1.1)}), there
exist a  projection $p'''\in A$, an  element $g_3\in A$ with $0\leq g_3\leq 1$,  and a $C^*$-subalgebra $D_2$ of $A$ with
$g_2\in D_2$, $1_{D_2}=p'''$, and $D_2\in \Omega$, and there exists a positive element $a^{(5)}\in D_2$ such that
 $$\|a''-a^{(5)}-(1-(p'''-g_3))^{1/2}a''(1-(p'''-g_3))^{1/2}\|<{\varepsilon},\hspace{0.8cm}(\textbf{3.10.4})  $$
  $$\langle (a^{(5)}-3{\varepsilon})_+\rangle\leq \langle  a \rangle,~~~ {\rm {and}}~~~\langle 1-(p'''-g_3) \rangle\leq \langle d \rangle.$$

   Since  $D_2$  is $m$-almost divisible, and $(a^{(5)}-3{\varepsilon})_+\in D_2$,  there exists  $b_4\in (D_2)_+$ such that $k\langle b_4 \rangle\leq \langle (a^{(5)}-3{\varepsilon})_+\rangle$,
 and $ \langle (a^{(5)}-4{\varepsilon})_+\rangle \leq (k+1)(m+1)\langle b_4 \rangle$.

   Then, as in the part  \textbf{(1.1.1)},
\begin{eqnarray}
\label{Eq:eq1}
  &&k\langle  (b_1\oplus b_4) \rangle= k\langle b_1 \rangle+k\langle b_4 \rangle\nonumber\\
  &&\leq \langle (a'-3{\varepsilon})_+ \rangle+\langle (a^{(5)}-3{\varepsilon})_+ \rangle \nonumber\\
  &&\leq \langle (a'-2{\varepsilon})_+ \rangle+\langle (a^{(5)}-3{\varepsilon})_+ \rangle \nonumber\\
&&\leq \langle  a \rangle+\langle  a \rangle,\nonumber
\end{eqnarray}

By $(\textbf{3.10.1})$ and $(\textbf{3.10.4})$,
 $$\|a-a'-a^{(5)}-(1-(p'''-g_3))^{1/2}a''(1-(p'''-g_3))^{1/2}\|<2{\varepsilon},$$
  and therefore, as in the part \textbf{(1.1.1)},
\begin{eqnarray}
\label{Eq:eq1}
  &&\langle (a-11\varepsilon)_+\rangle\nonumber\\
  &&\leq \langle (a'-5{\varepsilon})_+ \rangle +\langle (a^{(5)}-4{\varepsilon})_+  \rangle +\langle (1-(p'''-g_3))^{1/2}a''(1-(p'''-g_3))^{1/2}\rangle \nonumber\\
  &&\leq \langle  (a'-5{\varepsilon})_+ \rangle+\langle (a^{(5)}-4{\varepsilon})_+ \rangle  +\langle  1-(p'''-g_3)\rangle\nonumber\\
  &&\leq \langle (a'-5{\varepsilon})_+ \rangle+ \langle d\rangle+\langle (a^{(5)}-4{\varepsilon})_+\rangle  \nonumber\\
&&\leq \langle  (a'-4{\varepsilon})_+ \rangle +\langle  (a^{(5)}-4{\varepsilon})_+ \rangle\nonumber\\
&&\leq (k+1)(m+1)\langle b_1\oplus b_4\rangle.\nonumber
\end{eqnarray}

These are  the desired inequalities, with $ b_1\oplus b_4$ in place of $b$ and  $11\varepsilon$ in place of $\varepsilon$.

\textbf{ Case (2)} We assume that $(a'-3{\varepsilon})_+$ is not Cuntz equivalent to a projection.

 By $(3)$ of
 Theorem 2.1, there is a non-zero positive element $e$  such that
 $\langle  (a'-4{\varepsilon})_+ \rangle+ \langle e\rangle \leq \langle(a'-3{\varepsilon})_+ \rangle$.

As in the part  \textbf{(1.1.1)}, as $A\in \rm {WTA}\Omega$ (with $e$ in place of $c$ in the part  \textbf{(1.1.1)}), there exist a  projection $p''''\in A$, an  element $g_4\in A$ with $0\leq g_4\leq 1$,  and a $C^*$-subalgebra $D_3$ of $A$ with
$g\in D_3$, $1_{D_3}=p''''$, and $D_3\in \Omega$, and there exists a  positive element $a^{(6)}\in D_3$ such that
  $$\|a''-a^{(6)}-(1-(p''''-g_4))^{1/2}a''(1-(p''''-g_4))^{1/2}\|<{\varepsilon},
  \hspace{0.8cm}(\textbf{3.10.5})  $$
  $$\langle (a^{(6)}-3{\varepsilon})_+\rangle\leq \langle  a \rangle,~~~ {\rm {and}}~~~\langle 1-(p''''-g_4) \rangle\leq \langle e \rangle.$$

  Since  $D_3$  is   $m$-almost divisible, and $(a^{(6)}-3{\varepsilon})_+\in D_3$,  there exists  $b_5\in (D_3)_+$ such that $k \langle b_5 \rangle\leq \langle  (a^{(6)}-3{\varepsilon})_+ \rangle$,
 and $ \langle (a^{(6)}-4{\varepsilon})_+\rangle\leq (k+1)(m+1)\langle b_5\rangle$.

    Then, as in  the part \textbf{(1.1.1)},
\begin{eqnarray}
\label{Eq:eq1}
  &&k\langle  (b'\oplus b_5) \rangle= k\langle b' \rangle + k\langle b_5 \rangle\nonumber\\
  &&\leq \langle (a'-2{\varepsilon})_+ \rangle +\langle (a^{(6)}-3{\varepsilon})_+\rangle  \nonumber\\
&&\leq \langle  a \rangle+\langle  a \rangle,\nonumber
\end{eqnarray}

By $(\textbf{3.10.1})$ and $(\textbf{3.10.5})$,
 $$\|a-a'-a^{(6)}-(1-(p''''-g_4))^{1/2}a''(1-(p''''-g_4))^{1/2}\|<2{\varepsilon},$$
  and therefore,  as before,
\begin{eqnarray}
\label{Eq:eq1}
  &&\langle (a-10\varepsilon)_+ \rangle \nonumber\\
  &&\leq \langle  (a'-4{\varepsilon})_+ \rangle+ \langle (a^{(6)}-4{\varepsilon})_+ \rangle + \langle (1-(p''''-g_4))^{1/2}a''(1-(p''''-g_4))^{1/2} \rangle\nonumber\\
  &&\leq  \langle (a'-4{\varepsilon})_+ \rangle+ \langle (a^{(6)}-4{\varepsilon})_+  \rangle+ \langle 1-(p''''-g_4) \rangle\nonumber\\
&&\leq \langle (a'-4{\varepsilon})_+ \rangle+\langle e \rangle+\langle (a^{(6)}-4{\varepsilon})_+ \rangle \nonumber\\
&&\leq \langle (a'-3{\varepsilon})_+ \rangle+ \langle (a^{(6)}-4{\varepsilon})_+ \rangle \nonumber\\
&&\leq  (k+1)(m+1)\langle b'\oplus b_5 \rangle.\nonumber
\end{eqnarray}

These are  the desired inequalities, with $b'\oplus b_5$ in place of $b$ and  $10\varepsilon$ in place of $\varepsilon$.
\end{proof}

The proof of Theorem \ref {thm:3.11} which follows  is similar to that of  Theorem \ref {thm:3.10}.

 \begin{theorem}\label{thm:3.11} Let $\Omega$ be a class of  unital
${\rm C^*}$-algebras  which are  $m$-almost divisible (Definition \ref {def:2.7}).  Let $A\in{\rm TAA}\Omega$ be a simple unital stably finite ${\rm C^*}$-algebra such that for any $n\in \mathbb{N}$ and any  unital hereditary
 ${\rm C^*}$-subalgebra $D$ of  $\rm{M}$$_n(A)$, $D$  belongs to the class ${\rm TAA}\Omega$. Then
   $A$  is weakly $(2, m)$-almost divisible (Definition \ref {def:2.8}).\end{theorem}
\begin{proof} We must show that there is $b\in {\rm M}_{\infty}(A)_+$  such that  $k\langle b\rangle \leq  \langle a \rangle+\langle a \rangle$  and $\langle(a-\varepsilon)_+\rangle\leq (k+1)(m+1)\langle b\rangle$,
 for any given  $a\in A_+,$  $\varepsilon>0$, and $k\in \mathbb{N}.$
  We may assume that $\|a\|=1.$ (We have replaced $\rm {M}$$_n(A)$ containing $a$ given initially  by $A$.)

For any  $\delta_2>0$,  since   $A\in{\rm  TAA}\Omega$,  there
exist a  projection $p\in A$,   a $\rm C^*$-subalgebra $B$ of $A$ with
$1_B=p$ and $B\in \Omega$, and  a projection $q\in B$  such that

$(1)$ $qa\in_{{\delta_2/3}} B,~ aq\in_{{\delta_2/3}} B$.

By $(1)$,  $a-(1-q)a(1-q)=(qa+aq-qaq)\in_{{\delta_2}} B$, i.e., there exists $\bar{a}\in B$ such that $\|a-(1-q)a(1-q)-\bar{a}\|<\delta_2$. With sufficiently small $\delta_2$, by the functional calculus, we may assume that  there  exists a  positive element $a'\in B$ such that
 $$\|a-(1-q)a(1-q)-a'\|<\varepsilon. \hspace{0.8cm}(\textbf{3.11.1})$$

 Since  $B$ is   $m$-almost divisible, and $(a'-3\varepsilon)_+\in B$,  there exists  $b_1\in B_+$ such that $k\langle b_1\rangle \leq \langle(a'-3\varepsilon)_+\rangle$
 and $ \langle(a'-4\varepsilon)_+\rangle\leq (k+1)(m+1)\langle b_1\rangle$.

 Since  $B$ is   $m$-almost divisible, and $(a'-2\varepsilon)_+\in B$,  there exists  $b'\in B_+$ such that $k\langle b'\rangle\leq \langle(a'-2\varepsilon)_+\rangle$
 and $ \langle(a'-3\varepsilon)_+\rangle\leq (k+1)(m+1)\langle b'\rangle$.

We  divide the proof into two cases.

\textbf{Case (1)}  We assume that $(a'-3\varepsilon)_+$ is Cuntz equivalent to a projection.

\textbf{(1.1)} We assume that $(a'-4\varepsilon)_+$ is Cuntz equivalent to a projection.

\textbf{(1.1.1)}  If  $\langle(a'-4\varepsilon)_+\rangle$, the class of a projection,  is not equal to  $(k+1)(m+1)\langle b_1\rangle$, by Theorem 2.1 (2), there exists a  non-zero $c\in A_+$ such that $\langle (a'-4\varepsilon)_+\rangle+ \langle c\rangle \leq  (k+1)(m+1)\langle b_1\rangle$.

 For  any  $\delta_2>0$,  since   $(1-q)A(1-q)\in{\rm  TAA}\Omega$ (recall that  the present $A$ was initially $\rm {M}$$_n(A)$),  there
exist a  projection $p'\in (1-q)A(1-q)$,   a $\rm C^*$-subalgebra $D$ of $(1-q)A(1-q)$ with
$1_D=p'$ and $D\in \Omega$, and  a projection $q'\in D$  such that

$(1)'$ $q'(1-q)a(1-q) \in_{\delta_2/3}D,~ (1-q)a(1-q)q'\in_{\delta_2/3} D$, and

$(2)'$ $1-q-q'\precsim c$.


 By $(1)'$,  as above,  there exists a positive element $a'''\in B$  such that
 $$\|(1-q)a(1-q)-a'''-(1-q-q')(1-q)a(1-q)(1-q-q')\|<\varepsilon. \hspace{0.4cm}(\textbf{3.11.2})$$

By $(\textbf{3.11.1})$ and $(\textbf{3.11.2})$, one has
$$\|a-a'-a'''-(1-q-q')(1-q)a(1-q)(1-q-q')\|<2\varepsilon. \hspace{0.4cm}(\textbf{3.11.3})$$

  Since  $D$ is  $m$-almost divisible, and $(a'''-3\varepsilon)_+\in D$,  there exists  $b_2\in D_+$ such that $k\langle b_2\rangle \leq \langle  (a'''-3\varepsilon)_+\rangle $
 and $ \langle (a'''-4\varepsilon)_+\rangle \leq  (k+1)(m+1)\langle b_2\rangle $.

Since $a'\leq a'+a'''+(1-q-q')(1-q)a(1-q)(1-q-q')$, by Theorem 2.1 (5), $\langle (a'-2{\varepsilon})_+\rangle\leq \langle (a'+a'''+(1-q-q')(1-q)a(1-q)(1-q-q')-2{\varepsilon})_+\rangle$, and by $(\textbf{3.11.3})$, and Theorem 2.1 (1), $\langle (a'-2{\varepsilon})_+\rangle\leq \langle (a'+a'''+(1-q-q')(1-q)a(1-q)(1-q-q')-2{\varepsilon})_+\rangle\leq \langle a\rangle$.
By   a similar  argument (cf. proof of case (1.1.1) of Theorem 3.10),
 $\langle (a'''-3{\varepsilon})_+\rangle \leq \langle a\rangle$.

 Therefore,  we have
\begin{eqnarray}
\label{Eq:eq1}
  &&k\langle b_1\oplus b_2\rangle=k\langle b_1\rangle +k\langle b_2\rangle\nonumber\\
  &&\leq \langle (a'-3\varepsilon)_+\rangle +\langle (a'''-3\varepsilon)_+\rangle \nonumber\\
  &&\leq \langle (a'-2\varepsilon)_+\rangle +\langle (a'''-3\varepsilon)_+\rangle \nonumber\\
  &&\leq \langle a\rangle+\langle a\rangle.\nonumber
\end{eqnarray}

Since $\|a-a'-a'''-(1-q-q')(1-q)a(1-q)(1-q-q')\|<2\varepsilon,$
therefore,  by Theorem 2.1 (1) and Theorem 2.5 (4) (twice),
\begin{eqnarray}
\label{Eq:eq1}
  &&\langle (a-10\varepsilon)_+\rangle \nonumber\\
  &&\leq \langle (a'-4\varepsilon)_+\rangle +\langle  (a'''-4\varepsilon)_+\rangle + \langle (1-q-q')(1-q)a(1-q)(1-q-q')\rangle\nonumber\\
 &&\leq  \langle (a'-4\varepsilon)_+\rangle +\langle  (a'''-4\varepsilon)_+\rangle + \langle 1-q-q'\rangle\nonumber\\
&&\leq  \langle (a'-4\varepsilon)_+\rangle +\langle  (a'''-4\varepsilon)_+\rangle +\langle c\rangle ~~(\rm {by} ~(2)')~ \nonumber\\
&&\leq (k+1)(m+1)\langle b_1\rangle + (k+1)(m+1)\langle b_2\rangle =(k+1)(m+1)\langle b_1\oplus b_2\rangle.\nonumber
\end{eqnarray}

These are  the desired inequalities, with $b_1\oplus b_2$ in place of $b$ and $10\varepsilon$ in place of $\varepsilon$.

\textbf{(1.1.2)} If $\langle (a'-4\varepsilon)_+\rangle$  is equal to  $(k+1)(m+1)\langle b_1\rangle$, so that $(k+1)(m+1)\langle b_1\rangle \leq \langle (a'-3\varepsilon)_+\rangle $, then, as $m+1\geq 2$,
we have  $k\langle b_1\oplus b_1 \rangle \leq \langle (a'-3{\varepsilon})_+ \rangle $. Also, $\langle(a'-4{\varepsilon})_+ \rangle +\langle  b_1 \rangle \leq (k+1)(m+1)\langle b_1\oplus b_1\rangle $.

As  in the part \textbf{(1.1.1)}, since $(1-q)A(1-q)\in \rm {TAA}\Omega$ (with $b_1$ in place of $c$ as in the part \textbf{(1.1.1)}),  there
exist a  projection $p''\in (1-q)A(1-q)$,   a $\rm C^*$-subalgebra $D_1$ of $(1-q)A(1-q)$ with
$1_{D_1}=p''$ and $D_1\in \Omega$, and  a projection $q''\in D_1$, and there exists a  positive element $a^{(4)}\in D_1$
such that
 $$\|(1-q)a(1-q)-a^{(4)}-(1-q-q'')(1-q)a(1-q)(1-q-q'')\|<\varepsilon, \hspace{0.4cm}(\textbf{3.11.4})$$
   $$\langle (a^{(4)}-3{\varepsilon})_+\rangle\leq \langle  a \rangle,~~~ {\rm {and}} ~~~\langle 1-q-q'' \rangle\leq \langle b_1 \rangle.$$

 Since  $D_1$ is  $m$-almost divisible, and $(a^{(4)}-3\varepsilon)_+\in D_1$,  there exists  $b_3\in (D_1)_+$ such that $k\langle b_2\rangle \leq \langle  (a^{(4)}-3\varepsilon)_+\rangle$,
 and $\langle (a^{(4)}-4\varepsilon)_+\rangle \leq  (k+1)(m+1)\langle b_3\rangle$.

Then, as in the part \textbf{(1.1.1)},
\begin{eqnarray}
\label{Eq:eq1}
  &&k\langle b_1\oplus b_1\oplus b_3\rangle =k\langle b_1\oplus b_1\rangle + k\langle  b_3\rangle \nonumber\\
  &&\leq \langle(a'-3\varepsilon)_+\rangle +\langle (a^{(4)}-3\varepsilon)_+\rangle \nonumber\\
  &&\leq \langle(a'-2\varepsilon)_+\rangle +\langle (a^{(4)}-3\varepsilon)_+\rangle \nonumber\\
   &&\leq\langle a\rangle+\langle a\rangle.\nonumber
\end{eqnarray}

By $(\textbf{3.11.1})$ and $(\textbf{3.11.4})$, one has
$$\|a-a'-a^{(4)}-(1-q-q'')(1-q)a(1-q)(1-q-q'')\|<2\varepsilon,$$
 and therefore, as in the part \textbf{(1.1.1)},
\begin{eqnarray}
\label{Eq:eq1}
  &&\langle(a-10\varepsilon)_+\rangle\nonumber\\
  &&\leq  \langle (a'-4\varepsilon)_+\rangle +\langle (a^{(4)}-4\varepsilon)_+\rangle + \langle(1-q-q'')(1-q)a(1-q)(1-q-q'')\rangle\nonumber\\
  &&\leq  \langle(a'-4\varepsilon)_+\rangle +\langle (a^{(4)}-4\varepsilon)_+\rangle + \langle 1-q-q''\rangle \nonumber\\
&&\leq  \langle (a'-4\varepsilon)_+\rangle +\langle (a^{(4)}-4\varepsilon)_+\rangle +\langle  b_1\rangle \nonumber\\
&&\leq  (k+1)(m+1)\langle b_1\oplus b_1\rangle+(k+1)(m+1)\langle b_3\rangle\nonumber\\
&&=(k+1)(m+1)\langle b_1\oplus b_1\oplus b_3\rangle.\nonumber
\end{eqnarray}

These are  the desired inequalities, with $ b_1\oplus b_1\oplus b_3$ in place of $b$ and $10\varepsilon$ in place of $\varepsilon$.

\textbf{(1.2)} We assume that $ (a'-4\varepsilon)_+$ is not Cuntz equivalent to a projection.
 By  Theorem 2.1 (3), there is a non-zero positive element $d$  such that
 $\langle (a'-5\varepsilon)_+\rangle +\langle d\rangle \leq \langle (a'-4\varepsilon)_+\rangle$.

As  in the part \textbf{(1.1.1)}, since $(1-q)A(1-q)\in \rm {TAA}\Omega$ (with $d$ in place of $c$ as in the part \textbf{(1.1.1)}), there
exist a projection $p'''\in (1-q)A(1-q)$,   a $\rm C^*$-subalgebra $D_2$ of $(1-q)A(1-q)$ with
$1_{D_2}=p'''$ and $D_2\in \Omega$, and  a projection $q'''\in D_2$, and there exists a  positive element $a^{(5)}\in D_2$,
such that
 $$\|(1-q)a(1-q)-a^{(5)}-(1-q-q''')(1-q)a(1-q)(1-q-q''')\|<\varepsilon, \hspace{0.4cm}(\textbf{3.11.5})$$
   $$\langle (a^{(5)}-3{\varepsilon})_+\rangle\leq \langle  a \rangle,~~~ {\rm {and}} ~~~\langle 1-q-q''' \rangle\leq \langle d \rangle.$$

 Since  $D_2$ is  $m$-almost divisible, and $(a^{(5)}-3\varepsilon)_+\in D_2$,  there exists  $b_4\in (D_2)_+$ such that $k\langle b_4\rangle \leq \langle  (a^{(5)}-3\varepsilon)_+\rangle $
 and $\langle (a^{(5)}-4\varepsilon)_+\rangle \leq (k+1)(m+1)\langle b_4\rangle$.

As in  the part \textbf{(1.1.1)}, one has
\begin{eqnarray}
\label{Eq:eq1}
  &&k\langle b_1\oplus b_4\rangle = k\langle b_1\rangle + k\langle b_4\rangle\nonumber\\
  &&\leq \langle (a'-3\varepsilon)_+\rangle +\langle (a^{(5)}-3\varepsilon)_+\rangle \nonumber\\
   &&\leq \langle (a'-2\varepsilon)_+\rangle +\langle (a^{(5)}-3\varepsilon)_+\rangle \nonumber\\
   &&\leq \langle a\rangle+\langle a\rangle.\nonumber
\end{eqnarray}

By $(\textbf{3.11.1})$ and $(\textbf{3.11.5})$, one has
$$\|a-a'-a'''-(1-q-q''')(1-q)a(1-q)(1-q-q''')\|<2\varepsilon,$$
and therefore, as in the part \textbf{(1.1.1)},
\begin{eqnarray}
\label{Eq:eq1}
  &&\langle(a-11\varepsilon)_+\rangle\nonumber\\
  &&\leq \langle (a'-5\varepsilon)_+\rangle +\langle  (a^{(5)}-4\varepsilon)_+\rangle + \langle (1-q-q')(1-q)a(1-q)(1-q-q')\rangle \nonumber\\
  &&\leq \langle (a'-5\varepsilon)_+\rangle +\langle  (a^{(5)}-4\varepsilon)_+\rangle + \langle 1-q-q'\rangle \nonumber\\
&&\leq \langle (a'-5\varepsilon)_+\rangle +\langle (a^{(5)}-4\varepsilon)_+ \rangle +\langle d\rangle\nonumber\\
&&\leq  \langle (a'-4\varepsilon)_+\rangle +\langle  (a^{(5)}-4\varepsilon)_+ \rangle \nonumber\\
&&\leq (k+1)(m+1)\langle b_1\rangle + (k+1)(m+1)\langle b_4\rangle = (k+1)(m+1)\langle b_1\oplus b_4\rangle.\nonumber
\end{eqnarray}

These are  the desired inequalities, with $b_1\oplus b_4$ in place of $b$ and  $11\varepsilon$ in place of $\varepsilon$.

\textbf{Case (2)} We assume that  $(a'-3\varepsilon)_+$ is not Cuntz equivalent to a projection.

 By Theorem 2.1 (3), there is a non-zero positive element $e$  such that
 $\langle (a'-4\varepsilon)_+\rangle +\langle e\rangle \leq \langle (a'-3\varepsilon)_+\rangle$.

As  in the part \textbf{(1.1.1)},  since $(1-q)A(1-q)\in \rm {TAA}\Omega$ (with $e$ in place of $c$ as in the part \textbf{(1.1.1)}), there
exist a  projection $p''''\in (1-q)A(1-q)$,   a $\rm C^*$-subalgebra $D_3$ of $(1-q)A(1-q)$ with
$1_{D_3}=p''''$ and $D_3\in \Omega$, and  a projection $q''''\in D_3$, and there exists a  positive element $a^{(6)}\in D_3$,
such that
 $$\|(1-q)a(1-q)-a^{(6)}-(1-q-q')(1-q)a(1-q)(1-q-q')\|<\varepsilon, \hspace{0.4cm}(\textbf{3.11.6})$$
   $$\langle (a^{(6)}-3{\varepsilon})_+\rangle\leq \langle  a \rangle,~~~ {\rm {and}} ~~~\langle 1-q-q'''' \rangle\leq \langle e \rangle.$$

 Since  $D_3$ is  $m$-almost divisible, and $(a^{(6)}-3\varepsilon)_+\in D_3$,  there exists  $b_5\in (D_3)_+$ such that $k\langle b_5\rangle \leq \langle (a^{(6)}-3\varepsilon)_+\rangle$
 and $\langle (a^{(6)}-4\varepsilon)_+\rangle \leq  (k+1)(m+1)\langle b_2\rangle$.

As in the part \textbf{(1.1.1)},
\begin{eqnarray}
\label{Eq:eq1}
  &&k\langle b_1'\oplus b_5\rangle =k\langle b_1'\rangle + k\langle b_5\rangle \nonumber\\
  &&\leq \langle (a'-2\varepsilon)_+\rangle +\langle (a^{(6)}-3\varepsilon)_+\rangle  \nonumber\\
 &&\leq \langle a\rangle+\langle a\rangle.\nonumber
\end{eqnarray}

By $(\textbf{3.11.1})$ and $(\textbf{3.11.6})$, one has
$$\|a-a'-a^{(6)}-(1-q-q'''')(1-q)a(1-q)(1-q-q'''')\|<2\varepsilon,$$
and therefore, as before,
\begin{eqnarray}
\label{Eq:eq1}
  &&\langle (a-10\varepsilon)_+\rangle \nonumber\\
  &&\leq \langle (a'-4\varepsilon)_+\rangle +\langle (a^{(6)}-4\varepsilon)_+\rangle + \langle (1-q-q'''')(1-q)a(1-q)(1-q-q'''')\rangle \nonumber\\
  &&\leq  \langle (a'-4\varepsilon)_+\rangle +\langle (a^{(6)}-4\varepsilon)_+\rangle + \langle 1-q-q''''\rangle \nonumber\\
&&\leq \langle (a'-4\varepsilon)_+\rangle +\langle(a^{(6)}-4\varepsilon)_+ \rangle +\langle e \rangle\nonumber\\
&&\leq \langle (a'-3\varepsilon)_+\rangle +\langle  (a^{(6)}-4\varepsilon)_+\rangle  \nonumber\\
&&\leq (k+1)(m+1)\langle b_1'\rangle + (k+1)(m+1)\langle b_5\rangle =(k+1)(m+1)\langle b_1'\oplus b_5\rangle.\nonumber
\end{eqnarray}

These are  the desired inequalities, with  $b_1'\oplus b_5$ in place of $b$ and  $10\varepsilon$ in place of $\varepsilon$.
\end{proof}

\emph{\textbf{Acknowledgements:}} The  research of the first   author was    supported by a grant from  the Natural Sciences and Engineering Research Council of Canada. The research of the third   author was  supported by a grant from the National Natural Sciences Foundation of China (Grant No. ~ 11871375).

 \end{document}